\documentclass[11pt]{article}
\usepackage{amsmath, amssymb, amscd, amsthm, amsfonts, subfig}
\usepackage{graphicx}
\usepackage{hyperref}
\usepackage{verbatim}
\usepackage{float}
\usepackage{mathtools}
\usepackage{color,colortbl}

\usepackage[font=small,labelfont=bf]{caption}
\usepackage{titlesec}
\expandafter\def\expandafter\normalsize\expandafter{%
    \normalsize%
    \setlength\abovedisplayskip{5pt}%
    \setlength\belowdisplayskip{5pt}%
    \setlength\abovedisplayshortskip{-3pt}%
    \setlength\belowdisplayshortskip{3pt}%
}
\usepackage[scr=boondox, 
            cal=esstix]   
           {mathalpha}
\newtheorem{theorem}{Theorem}[section] 
\newtheorem{corollary}{Corollary}[theorem] 
\newtheorem{lemma}[theorem]{Lemma} 
\newtheorem{proposition}[theorem]{Proposition}

\theoremstyle{definition}
\newtheorem{example}{Example}

\theoremstyle{remark}
\newtheorem*{remark}{Remark} 

\DeclareMathOperator{\outdeg}{outdegree}

\newcommand{\floor}[1]{\left\lfloor #1 \right\rfloor}

\title{Labeled Chip-Firing on Directed $k$-ary Trees and Where Chips Land}
\author{Ryota Inagaki \and Tanya Khovanova \and Austin Luo}
\date{}

\begin{document}

\maketitle

\begin{abstract}
    Chip-firing is a combinatorial game played on a graph, in which chips are placed and dispersed on the vertices until a stable configuration is achieved. We study a chip-firing variant on an infinite, rooted directed $k$-ary tree, where we place $k^n$ chips labeled $1,2,3,\dots, k^n$ on the root for some nonnegative integer $n$. A vertex $v$ can fire if it has at least $k$ chips; when it fires, $k$ chips are selected, and the chip with the $i$th smallest label is sent to the $i$th leftmost child of $v$. A stable configuration is reached when no vertices can fire. In this paper, we prove numerous properties of the stable configuration, such as that chips land on vertices in ranges and the lengths of those ranges. We also describe where each chip can land. This helps us describe possible stable configurations of the game.
\end{abstract}

\renewcommand{\thefootnote}{\fnsymbol{footnote}} 



\section{Introduction}
Chip-firing is a single-player game on graphs and is an important topic in algebraic, enumerative, and structural combinatorics. Chip-firing originates from problems such as the Abelian Sandpile studied by Bak et al.~\cite{PhysRevLett} and Dhar \cite{dhar1999abelian}, which states that when a stack of sand grains exceeds a certain height, the stack will disperse grains evenly to its neighbors. Eventually, the state of the sandpile may be a stable configuration, a distribution of sand in which every stack of sand grains cannot reach the threshold to disperse. The study of chip-firing as a combinatorial game on graphs began from works such as those of Spencer \cite{MR856644}, Anderson et al.~\cite{zbMATH04135751} and Bj\"orner et al.~\cite{MR1120415}. Many variants of the chip-firing game (see, for instance, \cite{MR4486679, MR3311336, MR3504984}) allow the discovery of different properties. In each variant, chips are either distinguishable (labeled) or indistinguishable (unlabeled) from each other.

\subsection{Unlabeled Chip-Firing on Directed Graphs}
\label{sec:unlabeledchipfiring}

Unlabeled chip-firing is a game where indistinguishable chips are placed on vertices in a directed graph $G = (V, E)$. If a vertex has enough chips to transfer one chip to each out-neighbor, then that vertex can fire. If there are at least $\outdeg(v)$ chips at a vertex $v$, it can fire, that is, it sends one chip to each neighbor and thus loses $\outdeg(v)$ chips. Once no vertex can fire, we reach a \textit{stable configuration} of chips.

\begin{example}
Figure~\ref{fig:exampleunlabel} shows the unlabeled chip-firing process when we start with $4$ chips at the root of an infinite binary tree.
\end{example}

\begin{figure}[H]
\centering
    \subfloat[\centering Initial configuration with $4$ chips]{{\includegraphics[width=0.3\linewidth]{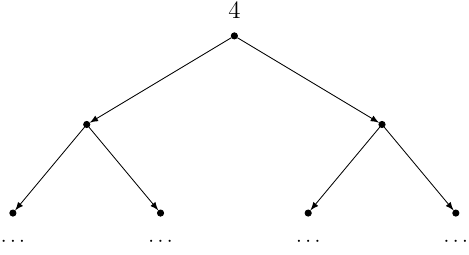} }}%
    \qquad
    \subfloat[\centering Configuration after firing once]{{\includegraphics[width=0.3\linewidth]{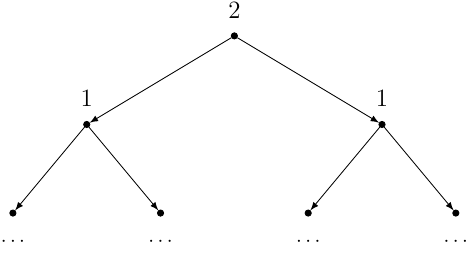} }}%
    \qquad
    \subfloat[\centering Configuration after firing twice]{{\includegraphics[width=0.3\linewidth]{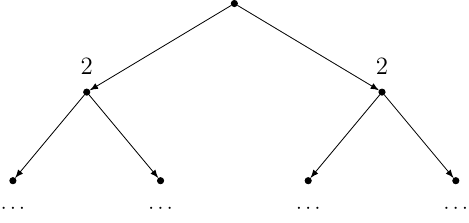} }}%
    \qquad
    \subfloat[\centering Stable configuration]{{\includegraphics[width=0.3\linewidth]{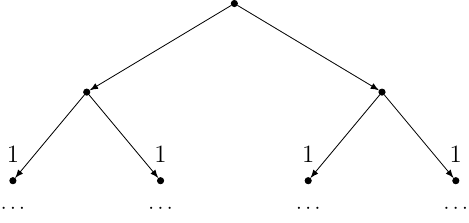} }}%
    \caption{Example of unlabeled chip-firing on an infinite directed, rooted binary tree}%
    \label{fig:exampleunlabel}
\end{figure}

We define a \textit{configuration} $\mathcal{C}$ as a distribution of chips over the vertices of a graph, which one can write as a vector $\vec{c}$ in $\mathbb{N}^{|V|}$ (the set of infinite sequences indexed by the nonnegative integers whose entries are nonnegative integers), where the $k$th entry in $\vec{c}$ is the number of chips on vertex $v_k$ of the graph. One important property of directed graph chip-firing with unlabeled chips is the following analog to confluence properties for undirected graph chip-firing (similar to Theorem 2.2.2 of \cite{klivans2018mathematics}) and stabilization of Gabrielov's Abelian Avalanche model \cite{ MR1215018, gabrielov1994asymmetric}.

\begin{theorem}[Theorem 1.1 of \cite{MR1203679} and Theorem 2.1 of \cite{MR1120415}]\label{thm:GlobalConfluence}
For a directed graph $G$ and an initial configuration $\mathcal{C}$ of chips on $G$, the unlabeled chip-firing game will either run forever or end after the same number of moves and at the same stable configuration. Also, the number of times each vertex fires is the same regardless of the sequence of fires taken in the game.
\end{theorem}

\subsection{Labeled Chip-Firing on Directed Graphs}
\label{labeledchipfirng}

Labeled chip-firing is a variant of chip-firing where the chips are distinguishable. In this paper, we assign to each chip a number from the set $\{1,2,\dots, N\}$ where there are $N$ chips in total. A vertex $v$ can fire if it has at least $\outdeg(v)$ chips. When a vertex fires, we choose any $\outdeg(v)$ of the labeled chips and disperse them, one chip for each neighbor. The chip each neighbor receives may depend on the label of the chip. The study of labeled chip-firing was initiated in the context of one-dimensional lattices by Hopkins, McConville, and Propp \cite{MR3691530}.

In this paper, we continue the study of labeled chip-firing specifically in the context of infinite directed $k$-ary trees for $k \geq 2$. Consider an infinite directed $k$-ary tree, more specifically, the set of possible chips that can land at the vertices of the tree. Since each vertex $v$ has $\outdeg(v) = k$, a vertex can fire if it has at least $k$ chips. When a vertex fires, we arbitrarily select $k$ chips and send the $i$th smallest chip to the $i$th leftmost child. Note that when we say a chip is smaller than or larger than another chip, we refer to the numerical values of the labels assigned to them. 

In labeled chip-firing, the global confluence property does not always hold. In other words, one can obtain different stable configurations depending on the sets of chips we arbitrarily select to fire. Two stable configurations would always have the same number of chips at each vertex, but the labels might differ. This prompts an exciting branch of research in chip-firing, even in the context of simple settings (e.g., \cite{MR4333108, MR4827886}).
 
\begin{example}
Consider a directed binary tree where each vertex has two children and an initial configuration consisting of the labeled chips $1,2,3,4$ at the root. Notice that since chips are only sent along directed edges, once a chip is sent to the left or right, it cannot go back. Therefore, if we fire the pair of chips $(1,2)$ first, we end up with a different stable configuration than if we fire the pair $(2,3)$ first. Figure~\ref{fig:confluencebreak} illustrates this initial firing. 
\end{example}

\begin{figure}[H]
\centering
    \subfloat[\centering Configuration after firing $(1,2)$]{{\includegraphics[width=0.35\linewidth]{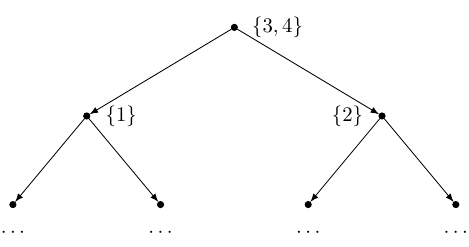} }}%
    \qquad
    \subfloat[\centering Configuration after firing $(2,3)$]{{\includegraphics[width=0.35\linewidth]{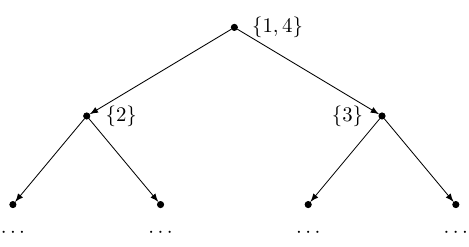} }}%
    \caption{Example of confluence breaking}%
    \label{fig:confluencebreak}
\end{figure}

Therefore, to obtain certain stable configurations, we pick certain sets of chips to fire.

 \subsection{Motivation}

Labeled chip-firing has been studied on infinite, undirected binary trees when starting with $2^{n} - 1$ chips at the root for some $n \in \mathbb{N}$ (where $0 \in \mathbb{N}$) by Musiker and Nguyen \cite{MR4827886}. At the end of their paper, Musiker and Nguyen posed the following questions:
 \begin{itemize}
     \item What do the possible stable configurations look like?
     \item How many of them are there?
 \end{itemize}
 
In \cite{inagaki2024chipfiringundirectedbinarytrees}, the authors of this paper partially answered the second question above by providing an upper bound on the number of possible stable configurations. In  \cite{MR4887467}, the authors of this paper addressed the above two questions but in the context of labeled chip-firing on directed $k$-ary trees by calculating the number of possible stable configurations of $k^n$ labeled chips on the directed $k$-ary tree and studying the stable configurations from the game as permutations of $1, 2, \dots, k^n$. In \cite{inagaki2025permutationbasedstrategieslabeledchipfiring}, the authors of this paper continued the study of stable configurations from the chip-firing game on directed $k$-ary trees by considering placements of chips resulting from firings dictated by a permutation-based strategy, i.e., a procedure in which how each vertex fires depends on a permutation of $1,2,\dots, n$ and its layer/vertical position within the tree.

In this paper, we address the first question above for the labeled chip-firing game on directed $k$-ary trees. We do this by answering the following question: given $k^n$ labeled chips starting at the root and the label $c$ of a chip, at which vertices can chip $c$ end up in during the chip-firing game? To answer this question, we start with a related question: given a vertex in the terminal configuration, which chips can land there?

\subsection{Road Map}

In Section~\ref{sec:prelim}, we introduce preliminaries and definitions of labeled chip-firing on the $k$-ary tree, such as traversal strings, which bijectively correspond to each vertex in the tree. Next, in Section~\ref{sec:landingorder}, we introduce a more precise way to track how chips move through the firing process. We do this via landing orders, which are ordered pairs of form $(t,x)$ where $t$ is a traversal string and the relative order $x$ of a chip on a vertex with traversal string $t$. In Section~\ref{sec:layer2}, we define a structure for a poset (partially ordered set) on landing orders for layer $2$ and prove formulas for the smallest and largest chips that can land in the landing order $(t,x)$ on layer $2$.

In Section~\ref{sec:smallestchip}, we use the results from Section~\ref{sec:layer2} to derive a formula for the smallest and largest chip value that can land on a vertex based on its traversal string. Furthermore, in Section~\ref{sec:ranges}, we extend the results in Section~\ref{sec:smallestchip} and prove that any chip in the range between the smallest and the largest can land on a vertex; we call such an interval a landing range. We describe several properties of landing ranges in Section~\ref{sec:propertiesofranges}, where we discuss possible lengths of landing ranges.

In Section~\ref{sec:WheredoChipsLand}, we look at where the chips can land from the point of view of a chip. We establish where chips can land and discuss the spreads of chips, the difference between the leftmost and rightmost vertices where they can land. We also show that chips can be grouped into sets of consecutive chips with the same landing pattern.

\section{Preliminaries and Definitions}
\label{sec:prelim}
We give definitions essential to our discussion of chip-firing. These are almost the same as those established in our previous paper \cite{MR4887467}.

\subsection{Definitions}
\label{sec:definitions}
In this paper, we consider infinite rooted directed $k$-ary trees as our underlying graphs. 

In a \textit{rooted tree}, we denote one distinguished vertex as the \textit{root} vertex $r$. Every vertex in the tree, excluding the root, has exactly one parent vertex. A vertex $v$ has \textit{parent} $v_p$ if there is a directed edge $v_p \to v$. If a vertex $v$ has parent $v_p$, then vertex $v$ is a \textit{child} of $v_p$. An \textit{infinite directed $k$-ary} tree is an infinite directed rooted tree where each vertex has an out-degree of $k$ and an in-degree of $1$, except the root, which has $k$ children but zero parents. The edges are directed from a parent to the children. 

We define a \textit{traversal string} $t$ to be a finite-length string of integers from $1$ to $k$ inclusive. Each traversing string $t = t_1t_2\dots t_i$ defines a directed path from the root vertex, where for each $j \in [i]$ the $(j+1)$st vertex in the path is the $t_j$th leftmost child of its parent. The vertex $v_t$ defined by the traversing string $t = t_1t_2\dots t_{i}$ is the unique vertex that is at the endpoint of this path.

\begin{example}
    Consider the directed 3-ary tree in Figure \ref{fig:3arytraversal}. The second leftmost vertex in layer 3, which is labeled $v_{12}$, has the traversal string $12$ since it is reached by taking the leftmost child of the root $v_1$ and the second leftmost child of vertex $v_1$. The vertex $v_{22}$ has traversal string $22$ since it is reached by traversing down a middle-directed edge twice.
\end{example}

\begin{figure}[H]
    \centering
    \includegraphics[width=0.8\linewidth]{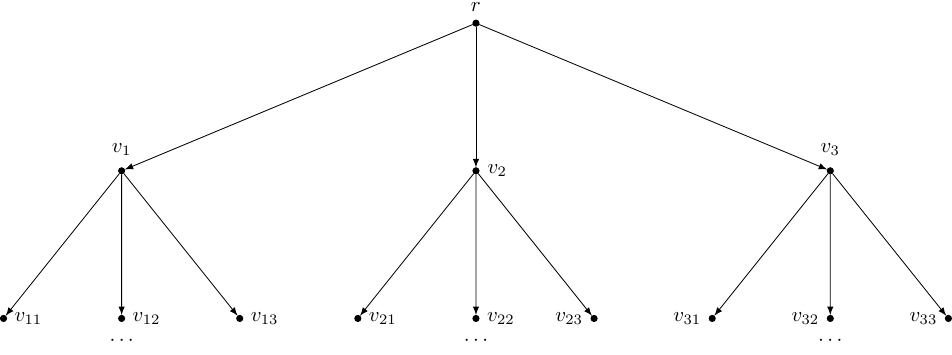}
    \caption{Traversal strings on layers of a $3$-ary tree }
    \label{fig:3arytraversal}
\end{figure}

We define the \textit{initial configuration} (or state) of chip-firing as placing $N$ chips on the root where, in the case of labeled chip-firing, they are labeled $1,2, \dots, N$. A vertex $v$ can \textit{fire} if it has at least $\outdeg(v)=k$ chips. When vertex $v$ \textit{fires}, it transfers a chip from itself to each of its $k$ neighbors. In the setting of labeled chip-firing, when a vertex fires, it chooses and fires $k$ of its chips so that among those $k$ chips, the one with the $ith$ smallest label gets sent to the $i$th leftmost child from the left.

A \textit{strategy} is an algorithm that dictates the order in which $k$-tuples of chips on a vertex are fired from each vertex.

In this paper, we assume $k \geq 2$ since if $k = 1$ and the tree has any positive number of chips, then the chip-firing process can continue indefinitely.

We define a vertex $v_j$ to be on \textit{layer} $i+1$ if the path of vertices traveled from the root to $v_j$ traverses $i$ vertices. Thus, the root $r$ is on layer $1$.

A \textit{stable configuration} is a distribution of chips over the vertices of a graph such that no vertex can fire. In this paper, we write each stable configuration as a permutation of $1, 2, \dots, k^{n}$, which is the sequence of chips in the $(n+1)$th layer of the tree in the stable configuration read from left to right. This is because, as we shall see in the next subsection, the stable configuration will have one chip at each vertex in layer $n+1$ since the chip-firing game always starts with $k^n$ chips at the root. This is our convention for the rest of this paper. For instance, the stable configuration in Figure~\ref{fig:directedex} is written as the permutation $135246789$.

Given a chip-firing game starting with $k^n$ labeled chips 1, 2, $\dots$, $k^n$ on the directed $k$-ary tree, we define a \textit{trivial chip} to be the chip labeled $1$ or $k^n$. Similarly, we define a \textit{nontrivial chip} to be any other chip. In this context, we define a \textit{trivial vertex} to be the leftmost or the rightmost vertex in the $(n+1)$st layer of the $k$-ary tree. We call the corresponding traversal string the \textit{trivial string}: these are the strings consisting of only $1$s or the string consisting only of $k$. We define a \textit{nontrivial vertex} to be any other vertex in the $n+1$st layer. We define a \textit{nontrivial traversal} string correspondingly.

\subsection{Unlabeled Chip-Firing on Directed \texorpdfstring{$k$}{k}-ary Trees}
\label{sec:unlabeled}

We first examine properties of unlabeled chip-firing on infinite directed $k$-ary trees when starting with $k^n$ chips at the root, where $n \in \mathbb{N}$. As the stable configuration and the number of firings do not depend on the order of firings, we can assume that we start from layer 1 and proceed by firing all the chips on the given layer before going to the next layer. Thus, for each $i \in \{1,2, \dots, n \}$, each vertex on layer $i$ fires $k^{n-i}$ times and sends $k^{n-i}$ chips to each of its children. In the stable configuration, each vertex on layer $n + 1$ has exactly $1$ chip, and for all $i \neq n + 1$, the vertices on layer $i$ have $0$ chips.

\subsection{Labeled Chip-Firing on Directed \texorpdfstring{$k$}{k}-ary Trees}
\label{sec:labeled}

 We now give an example of a labeled chip-firing game on the directed $k$-ary tree for $k=3$.
 
\begin{example}
Consider again a directed ternary (3-ary) tree with $9$ labeled chips at the root. Figure~\ref{fig:directedex} shows a possible sequence of firings.
\end{example}
\begin{figure}[H]
    \centering
    \subfloat[\centering Initial configuration with $9$ chips]{{\includegraphics[width=0.4\linewidth]{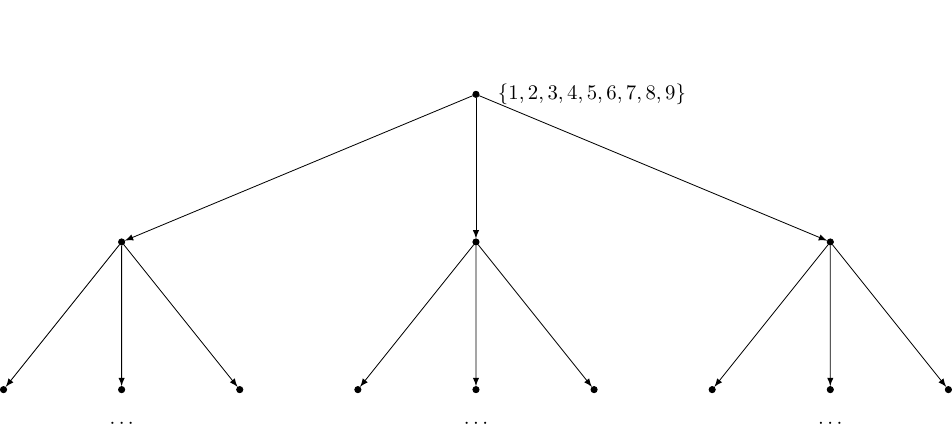} }}
    \qquad%
    \subfloat[\centering State after firing triples: $(1,2,7)$, $(3,4,8)$, and $(5,6,9)$ at the root]{{\includegraphics[width=0.4\linewidth]{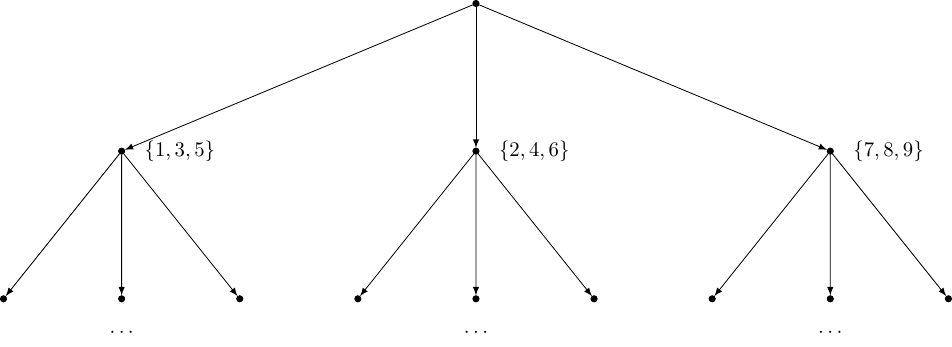} }}%
    \qquad
    \subfloat[\centering Stable configuration]
    {{\includegraphics[width=0.4\linewidth]{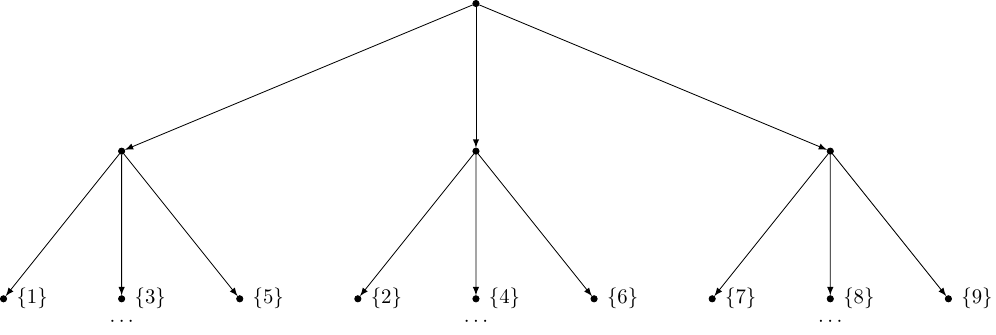} }}%
    \caption{Example of labeled chip-firing in a directed ternary tree with $9$ chips}
    \label{fig:directedex}
\end{figure}

In the previous example, observe that given that a vertex fires a set of ordered triples of labeled chips, any order in which those same triples of chips are fired yields the same distribution of chips to the children. This is a general fact, where we can replace triples with $k$-tuples.

We conclude the section with a useful lemma on the vertices where trivial chips end up: 
\begin{lemma}[Lemma 2.1 of \cite{MR4887467}]
 \label{lem:TrivialChips}
    Consider a directed $k$-ary tree with $k^n$ labeled chips starting at the root. Chip $1$ always lands on $v_{11\dots 11}$, and chip $k^n$ lands on  $v_{kk\dots kk}$. In addition, no nontrivial chip lands on trivial vertices.
\end{lemma}

\section{Landing Order}
\label{sec:landingorder}
\subsection{Definitions}

In our previous papers \cite{inagaki2024chipfiringundirectedbinarytrees, MR4887467, inagaki2025permutationbasedstrategieslabeledchipfiring}, we only considered the vertices the chips traversed. Now, we want to have a more detailed picture. We want to remember not only which vertex each chip passes through but also the relative order of that chip at that vertex.

We assume that during our firing process, we start with firing the root. When all chips at the root are fired, we continue with the next layer, and so on. For a given firing strategy, when the layer $i+1$ is full, suppose chip $c$ ends on a vertex $v_t$ on layer $i+1$ in a relative order $x$. Here $t=t_1t_2\dots t_i$ is the traversal string, and $x$ can have a value between 1 and $k^{n+1-i}$. Then, we define the \textit{landing order} of $c$ on layer $i+1$ as a pair $(t,x)$. For layer 1, we can assume that the traversal string is empty. Given a firing strategy, we denote a chip that ends in landing order $(t,x)$ as $c_{(t,x)}$.

\begin{example}
    For $k=n=2$, each layer has 4 landing orders. On layer 2, the orders are $(1,1)$, $(1,2)$, $(2,1)$, and $(2,2)$. On layer 3, the orders are $(11,1)$, $(12,1)$, $(21,1)$, and $(22,1)$.
\end{example}

We can envision that landing orders are arranged from left to right. Namely, the order $(t_1,x_1)$ is to the left from $(t_2,x_2)$ if and only if $t_1 < t_2$, meaning that vertex $v_{t_1}$ is to the left of $v_{t_2}$, or $t_1 = t_2$ and $x_1 < x_2$, that is, chips at the same vertex are ordered too.

If we number the landing orders from left to right, starting from index 1, we get a number for the landing order $(t,x)$, which we denote as $|(t,x)|$. The number $|(t,x)|$ can be calculated as follows: Subtract one from every digit in $t$, evaluate the result as a $k$-ary number, multiply by $k^{n+1-i}$, where $i+1$ is the layer, then add $x$. 
 
\begin{lemma}
\label{lem:landingordertonumber}
Given $i$ is the length of a traversing string $t$, the index of the landing order $(t, x)$ is $|(t, x)| = \sum_{\ell=1}^i (t_{\ell}-1)k^{n+1-\ell}+x$.
\end{lemma}

\begin{proof}
    Let $t=t_1t_2\dots t_i$ be a traversal string. First, we prove via induction on the length $i$ of $t$ that $v_{t}$ is the ($\sum_{\ell=1}^i (t_{\ell}-1)k^{i-\ell}+1$)st leftmost vertex in layer $i+1$.

    For our base case, let the length of the traversing string $t$ be $i = 1$. We know that $v_{t}$ is the $t_1$th leftmost vertex in layer $i+1=2$. The formula gives us
    \[\sum_{\ell=1}^1 (t_{\ell}-1)k^{i-\ell}+1 = (t_1 - 1)k^{1-1}+1 = t_1,\] confirming the base case.
    
    For our inductive step, we assume that $v_{t}$ is the ($\sum_{\ell=1}^i (t_{\ell}-1)k^{i-\ell}+1$)th leftmost vertex in layer $i+1$. There are exactly $k\sum_{\ell=1}^i (t_{\ell}-1)k^{i-\ell}$ vertices in layer $i+2$ that are children of vertices in layer $i+1$ but left of the children of $v_{t}$. Then observe that there are $(t_{i+1}-1)$ vertices that are children of $v_{t}$ and to the left of $v_{t_1t_2\dots t_it_{i+1}}$. Therefore there are a total of \[k\sum_{\ell=1}^i (t_{\ell}-1)k^{i-\ell} + (t_{i+1}-1) = \sum_{\ell=1}^{i+1} (t_{\ell}-1)k^{i+1-\ell}\]
    vertices to the left of $v_{t_1t_2\dots t_it_{i+1}}$. This completes the inductive step.

    Recall from Section \ref{sec:prelim} that for any $m \in \mathbb{N}$, the chip-firing process sends exactly $k^{n+1-m}$ chips to each vertex in layer $m$. From this fact along with the result from the previous paragraph, we obtain that for $t = t_1t_2\dots t_i$, there are exactly $k^{n+1-i}\sum_{\ell=1}^i(t_{\ell}-1)k^{i-\ell} = \sum_{\ell=1}^i (t_{\ell}-1)k^{n+1-\ell}$ chips that are to the left of any chips sent to $v_{t_1t_2\dots t_i}$. Then observe that there are $x-1$ chips to the left of the $x$th smallest chip on $v_t$. Thus in total there are exactly $\sum_{\ell=1}^i (t_{\ell}-1)k^{n+1-\ell}+x-1$ chips on layer $i+1$ that are before the $x$th smallest chip on $v_t$ in the landing order. Thus $|(t, x)| = \sum_{\ell=1}^i (t_{\ell}-1)k^{n+1-\ell}+x$.
\end{proof}

\begin{example}
    For $k=n=3$, consider the landing order $(31,2)$ on layer 3. Subtracting 1 from every digit of 31, we get 20. In ternary, it equals 6. Now, we are on layer 3, so we multiply by $3^{n+1-3}$ to get to 18. Then, we add $x=2$ and get that $|(31,2)| = 20$.
\end{example}

\subsection{Symmetries}

Our graph has a natural symmetry: reflection along a vertical line passing through the root. Let us denote the symmetric landing point to $(t,x)$ as $R(t,x)$. When we look at just the traversal string $t$, we denote its reflection as $R(t)$.

\begin{lemma}
    If $t$ is a traversal string to a vertex on layer $i+1$, then 
    \[R(t,x) = \left((k+1-t_1)(k+1-t_2)\dots (k+1-t_{i}), k^{n-i} + 1 -x\right).\]
\end{lemma}

\begin{proof}
    The traversal string that is symmetric to $t$ is $t' = (k+1-t_1)(k+1-t_2)\dots (k+1-t_{i})$. Each vertex on layer $i+1$ has $k^{n-i}$ orders. The relative order symmetric to $x$ is $k^{n-i} + 1 -x$.
\end{proof}

This symmetry means that if we can calculate the smallest chip that can land at a particular vertex, then we can also calculate the largest chip.

Let us denote the smallest chip that can land at $(t,x)$ as $a(t,x)$ and the largest as $b(t,x)$.

\begin{corollary}\label{cor:btx}
    We have
    \[b(t,x) = k^{n-i}+1 - a(R(t,x)).\]
\end{corollary}

When we start with $k^n$ chips at the root and $t$ is a traversal string of length $n$, we call the interval $[a(t, 1), b(t, 1)]$ the \textit{landing range} of chips that can end up at $v_t$.

\section{Layer 2}
\label{sec:layer2}
\subsection{Point-wise Domination}

We can define a structure of a poset on landing orders. We start with the partial order on traversal strings of the same length. We say that $t'$ dominates $t$, denoted as $t \preceq t'$, if and only if each digit of $t'$ is equal to or greater than the corresponding digit of $t$. We say that $(t',x')$ \textit{dominates} $(t,x)$, denoted as $(t,x) \preceq (t',x')$, if and only if $t \preceq t'$ and $x \leq x'$. The symbol $\prec$ is used to imply that the object on the right dominates the object on the left, and they are not equal. We say that a configuration of chips \textit{respects the domination order} if for each $(t, x)$ and $(t', x')$ such that $(t, x) \prec (t', x')$, then the chip that lands at $(t, x)$ is smaller than the chip that lands at $(t', x')$.

In general, if $(t,x) \prec (t',x')$, then the chip that lands at $(t, x)$ does not need to be smaller than the one that lands at $(t', x')$, as the following example shows.

\begin{example}
    Consider the stable configuration 13572648 for $k=2$ and $n=3$. It has chip 3 at the vertex $v_{221}$ and chip 4 at the vertex $v_{121}$. On the other hand, $(121, 1) \prec (221, 1)$. So, the terminal configuration does not respect the domination order.
\end{example}

However, chips on layer 2 respect the domination order.  Let $j$ and $j'$ be the traversal strings corresponding to the $j$th and $j'$th child of the root. 

\begin{proposition}
\label{prop:smallerchip}
    Let $j$ and $j'$ be traversal strings of length $1$. If $(j,x) \prec (j',x')$, then, for any firing strategy, the chip that lands at $(j, x)$ is smaller than the chip that lands at $(j',x')$:
    \[c_{(j,x)} < c_{(j',x')}.\]
\end{proposition}

\begin{proof}
Suppose at the $i$th firing of the root we fire $k$ chips $(a_{i, 1}, a_{i, 2}, \dots a_{i, k})$, where the set of chips is in order, that is for each $i \in [k^{n-1}]$ and $y < y'$, we have that $a_{i, y} < a_{i, y'}$. We can reorder $k$-tuples that are fired from the root in any way. Suppose we number $k$-tuples at the root by the order of chips landing at vertex $v_{j'}$. That is, we number them so that $a_{i, j'} = c_{(j, i')}$ for each $i \in [k^{n-1}]$.

Suppose $(j,x) \prec (j',x')$. We start with $j' = j$. Then, given a firing strategy $c_{(j,x)} < c_{(j,x')}$, by definition of the landing order.

Now, suppose $j < j'$. For any $j < j'$ and $x \leq x'$, we have
\[a_{x, j} \leq a_{x', j} < a_{x', j'}= c_{(j', x')}.\]
It follows that the vertex $v_{j}$ has at least $x'$ (and hence at least $x$) chips that are smaller than $c_{(t',x')}$, and the statement follows.
\end{proof}

The proposition above allows us to bound the smallest chip at a given landing on layer 2.

\begin{corollary}
\label{cor:smallestbound}
    If $(j,x)$ is a landing on layer 2, then the chip at this landing is at least $jx$:
    \[c_{(j,x)} \geq jx.\]
\end{corollary}

\begin{proof}
    If chip $c$ lands on Layer 2 at $(j, x)$, then the chip that lands at $(j',x') \prec (j,x)$ has a smaller value by Proposition~\ref{prop:smallerchip}. The number of such landings is $jx-1$.
\end{proof}

Moreover, if we have a distribution of chips on layer 2 that respects the domination order, it is possible to find a firing strategy that provides that order.

\begin{proposition}\label{prop:LandingOrderLayer2}
    If the distribution of chips on layer 2 respects the domination order, then a firing strategy at the root providing this distribution of chips exists.
\end{proposition}

\begin{proof}
    Consider a distribution of chips on layer 2 that respects the domination order. For a given $x$, consider a $k$-tuple of chips that are in the order $(j,x)$ for any $j$. Here, $j$ is the traversal string corresponding to the $j$th child of the root. Now, we take the set of chips in this $k$-tuple and use it in one fire at the root. After this fire, the chip of order $j$ in this tuple will go to the $j$th child. Thus, this chip goes to the correct child. We arrange all fires at the root like this, and, as a result, each child of the root gets the set of chips that was given in the distribution.
\end{proof}

The proposition allows us to show that a particular chip $c$ can land at $(j,x)$ on layer 2 by providing a distribution of chips that respects the domination order with chip $c$ at $(j,x)$.

\begin{corollary}
\label{cor:smallestexist}
    Chip $jx$ can land at $(j,x)$.
\end{corollary}

\begin{proof}
    Consider the following distribution of chips. Put chips from $1$ to $jx-1$ inclusive at the landing orders dominated by $(j,x)$ from left to right. Put chip $jx$ at the landing $(j,x)$. Put the rest of the chips on the remaining orders from left to right. This distribution of chips respects the domination order and, by Proposition~\ref{prop:LandingOrderLayer2}, a firing strategy at the root providing this distribution of chips exists.
\end{proof}

Now we are ready for our bounds.

\begin{proposition}
\label{prop:Layer2}
    The smallest chip that can land in the landing order $(t,x)$ on layer 2 is 
    \[a(j,x) = jx\]
    and the largest is
    \[b(j,x) = k^n+1 -(k^{n-1} +1-x)(k+1-j).\]
\end{proposition}

\begin{proof}
    Combining Corollary~\ref{cor:smallestbound} and~\ref{cor:smallestexist} we get that $a(j,x) = jx$.

    We can find $b(j, x)$ by applying Corollary \ref{cor:btx}, which tells us that $b(j, x) = k^n+1-a(R(j, x)) = k^n+1-a(k+1-j, k^{n-1}+1-x)$.
\end{proof}

\section{Smallest Chip at a Given Landing}
\label{sec:smallestchip}
Now, we calculate the smallest chip that can land at a particular landing order. By symmetry, we get the largest chip. We start with notation. We will need a product of the digits in the traversal string $t = t_1 t_2 \ldots t_i$. To make sure we never confuse the product with the string itself, we will always use $\prod$ for the product.

\begin{lemma}
\label{lem:smalllargechip}
    The smallest chip $c$ that can land in the landing order $(t,x)$, where $t = t_1t_2 \dots t_i$ is 
    \[a(t,x) = x \prod_{j=1}^i t_j,\]
    while the largest chip is 
    \[b(t,x) = k^n +1 -  (k^{n-i} +1 - x) \prod_{j=1}^i (k+1-t_j).\]
\end{lemma}
\begin{proof}    
    For layer 2, we proved this in Proposition~\ref{prop:Layer2}. Now we use induction on the layer number $i$. Suppose the lemma is true for layer $i$. Consider layer $i+1$. Let $t = t_1t_2 \dots t_i$ and $t' = t_1t_2 \dots t_i t_{i+1}$. Consider a subtree starting at $v_t$. In this subtree, by Proposition~\ref{prop:Layer2}, the smallest possible chip at the landing order $(t_{i+1},x)$ is the $(xt_{i+1})$th smallest chip at $v_t$. Thus, this chip is
    \[a(t,xt_{i+1}) = (xt_{i+1}) \prod_{j=1}^i t_j = x \prod_{j=1}^i t_j.\]

    To show that $b(t, x) = k^n +1-(k^{n-i}+1-x)\prod_{j=1}^i (k+1-t_j)$ we use Corollary~\ref{cor:btx} and the formula for $a(R(t, x))$.
\end{proof}

\begin{remark}
    The smallest chip at order $(t,x)$ is equal to the number of landing orders dominated by the order $(t,x)$.
\end{remark}

\begin{corollary}
    The smallest and the largest possible chips that land at vertex $v_t$ on the terminal layer $n+1$ are
    \[a(t, 1) = \prod_{j=1}^n t_j \quad \textrm{ and } \quad b(t, 1) = k^n+1-\prod_{j=1}^n(k+1-t_j),\]
    correspondingly.
\end{corollary}

\begin{example}
\label{ex:endsIn112}
    For string $t = 111\ldots 112$, the smallest possible chip in the stable configuration is 2, and the largest possible chip is 
    \[b(t, 1)= k^n +1- (k+1-2)\prod_{j=1}^{n-1}(k+1-1) = k^{n-1}+1.\]
\end{example}

\begin{example}
    For $k=n=3$, the smallest chip that can end at vertex $v_{213}$ in the terminal configuration is $2\cdot 1 \cdot 3 = 6$, and the largest is $3^3 + 1 - 2 \cdot 3 \cdot 1 = 28 - 6 = 22$.
\end{example}

\begin{example}
    For $k=2$ and the vertex $v_t$ at the layer $n+1$, the smallest and the largest chips can be expressed in terms of the number of twos in the traversal string. Equivalently, this is the number of right children the traversal string passes through. Let us denote this number as $m$. The smallest chip at vertex $v_t$ is $2^m$, and the largest is $2^n + 1 - 2^{n-m}$.
\end{example}

\section{Landing Ranges}
\label{sec:ranges}

We know that the smallest child that can land at the landing order $(t,x)$ is $a(t,x)$, and the largest is $b(t,x)$. In this section, we show that any number in between $a(t,x)$ and $b(t,x)$ can land at $(t,x)$. We start with layer 2.

\begin{lemma}\label{lem:Ranges}
   Given a landing $(t,x)$ on layer 2 and chip $c$ such that $a(t,x) \leq c \leq b(t,x)$, there exists a firing strategy that places $c$ at $(t,x)$.
\end{lemma}

\begin{proof}
    Consider a landing $(t,x)$ and the chip $c$ that is in the range from $a(t,x)$ to $b(t,x)$. Now, we describe the distribution of chips with respect to the domination order so that chip $c$ is at $(t,x)$.

    Consider the distribution of chips resulting from the following procedure: In step 1, we place the chips that are less than $a(t,x)$ to orders dominated by $(t,x)$ so that they respect the domination order. In step 2, we do the same for chips that are greater than $b(t,x)$. We place them at the landings that dominate $(t,x)$. We place $c$ at the vertex $v_t$ in step 3. At this point, all landings at vertex $v_t$ are matched with chips, and chip $c$ lands in order $x$. In step 4, we distribute the rest of the chips at the leftover landings in order from left to right.

    We now show that this distribution of chips respects the domination order. It is enough to show that for any vertex $v_{t'}$, the landings are in order, and for any second coordinate $x'$, the landings for different vertices are in order.

    We already saw that if $t' = t$, then the chips are distributed in order. Suppose $t' < t$. Then the landings at $t'$ that do not exceed $x$ are occupied by chips not exceeding $a(t,x)$ in order by step 1. The rest is occupied by chips between $a(t,x)$ and $b(t,x)$ in order by step 4. Thus, all of them are in order. By symmetry, a similar argument works for $t' > t$.

    Consider $x' = x$, then the chips at different children are in order by steps 1 and 2. Suppose $x' \leq x$, then for $t' < t$, the chips are in order by step 1, and all of them are less than $a(t,x)$. For $t' > t$, the chips are in order by step 4, and all of them are greater than $a(t,x)$. Thus, the chips are in order. A similar argument works for $x' > x$.

    Thus, the distribution of chips respects the poset order, and by Proposition~\ref{prop:LandingOrderLayer2}, the firing strategy exists that distributes the chips this way. The result follows.
\end{proof}

\begin{theorem}
\label{thm:ranges}
    For any landing $(t,x)$ on any layer, any chip in the range from $a(t,x)$ to $b(t,x)$ can appear there. In particular, the set of chips that can land at a particular vertex $v_t$ on the last layer forms a consecutive landing range
    \[a(t, 1) = \prod_{j=1}^n t_j \quad \textrm{ to } \quad b(t, 1) = k^n+1-\prod_{j=1}^n(k+1-t_j)\] inclusive. 
\end{theorem}

\begin{proof}
    We already proved the base case for layer 2 in Lemma~\ref{lem:Ranges}. We proceed by induction.
    
    Assume that at layer $i+1$, the chips landing at landing order $(t,x)$ form the interval $[a(t,x), b(t,x)]$.

    We now prove that the same is true for layer $i+2$. Consider a landing order $(t',x')$ at layer $i+2$, where $t'=t_1t$, where $t$ is a traversal string of length $i$. We now consider a subtree starting at a child $v_{t_1}$ of the root. Relative to this subtree, our landing order is on layer $i+1$. Thus, by the induction hypothesis, our chip can only arrive from the interval of orders $(t',(a(t,x))$ to $(t',(b(t,x))$.
    
    What is left to prove is that chips that can land at this interval form an interval. Consider two consecutive landing orders $(t_1,x)$ and $(t_1,x+1)$ on layer 2. Assume that $n > 1$ and $x < k^{n-1}$, since otherwise $x+1$ would be greater than the number of chips in each vertex in layer 2 and $(t_1, x+1)$ would not be a valid landing order.  We know that 
    \[a(t_1,x) = xt_1 < a(t_1,x+1) = (x+1)t_1\]
    \[b(t_1,x) = k^n +1- (k^{n-1} + 1 - x) (k+1-t_1) < b(t_1,x+1) =k^n+1 - (k^{n-1} + 1 - x-1) (k+1-t_1).\]
    Now, we want to show that for $x \in \{1, 2, \dots, k^{n-1}-1\}$,
    \[a(t_1,x+1) \leq b(t_1,x) + 1,\]
    or equivalently,
     \begin{equation}\label{eq:SlightlyBoringInequality}
         (x+1)t_1 \leq k^n +1- (k^{n-1} + 1 - x) (k+1-t_1) + 1.
     \end{equation}
     First, we rearrange the right-hand side
     \begin{multline*}
        k^n +1- (k^{n-1} + 1 - x) (k+1-t_1) + 1 = k^n-k^n + k^{n-1} (t_1-1) -k-1+t_1 +x (k+1-t_1) + 2 = \\
        k^{n-1} (t_1-1) -k+ t_1 +x (k+1-t_1)+1,
     \end{multline*}
     implying that inequality (\ref{eq:SlightlyBoringInequality}) above is equivalent to
     \[(x+1)t_1 \leq k^{n-1} (t_1-1) -k+ t_1 +x (k+1-t_1)+1.\]
     Rearranging, we obtain the following
     \begin{equation}\label{eq:SimplerInequality}
         (2t_1-k-1) x+k-1 \leq (t_1-1) k^{n-1}.
     \end{equation}

     Suppose $t_1 \leq k-1$, then we have
     \[(2t_1-k-1) x+k-1 \leq (t_1-2) x+k-1 \leq (t_1-2)(k^{n-1} -1) + k-1).\]
    The right-hand side can be rearranged to
    \[(t_1 - 1)k^{n-1} - k^{n-1} + k  \leq  (t_1 - 1)k^{n-1}.\]
    Suppose $t_1 = k$, then we have 
     \[(2t_1-k-1) x+k-1 = (k-1) x+k-1 = (k-1)(x+1) \leq (k-1) k^{n-1} = (t_1-1) k^{n-1}.\]
     We showed that two consecutive landing orders at the same vertex overlap. It follows that the union of any number of consecutive landing orders at the same vertex form an interval.
\end{proof}

\begin{example}
    Suppose $k=2$. Consider a nontrivial traversal string $t$ on the last layer. That means $t$ has fewer than $n$ ones and fewer than $n$ twos. Then, the smallest chip that lands at $t$ is
    \[\prod_{j=1}^n t_j \leq 2^{n-1}.\]
    Similarly, we can estimate the largest chip
    \[2^n+1-\prod_{j=1}^n(2+1-t_j) > 2^n - 2^{n-1} = 2^{n-1}.\]
    It follows that any landing range for a nontrivial vertex on the last layer contains the chip $2^{n-1}$. Hence, any pair of such landing ranges overlaps, and any consecutive set of such landing ranges is a range.
\end{example}

As the smallest and the largest chips landing at a particular vertex depend only on the set of digits in the traversal string, we get the following corollary.

\begin{corollary}\label{cor:PermutingTheTraversalString}
     Two vertices that have the same set of digits in their traversal strings get the same set of chips.
\end{corollary}

\begin{example}
Table~\ref{tab:ranges3-3} shows the landing ranges for $k=n=3$, depending on the traversal string up to a digit permutation.

\begin{table}[h!]\caption{Values for $a(t, 1)$ and $b(t, 1)$ for $k=n=3$ \label{tab:FirstThing}}
\centering
\begin{tabular}{|l|l|l|l|l|l|l|l|l|l|l|}
\hline
$t$              & 111 & 112 & 113 & 122 & 123 & 133 & 222 & 223 & 233 & 333 \\ \hline
$a(t,1)$ & 1 & 2 & 3 & 4 & 6 & 9 & 8 & 12 & 18 & 27 \\ \hline
$b(t,1)$ & 1 & 10 & 19 & 16 & 22 & 25 & 20 & 24 & 26 & 27 \\ \hline
Length & 1 & 9 & 17  & 13  & 17 & 17 & 13 & 13 & 9 & 1 \\ \hline
\end{tabular}
\label{tab:ranges3-3}
\end{table}
As we expected, symmetric string pairs (111 and 333, 112 and 233, 113 and 133, 122 and 223) have the same range length. The other two strings, 123 and 222, are self-symmetric and do not have pairs.
\end{example}

\section{Properties of the Landing Ranges}
\label{sec:propertiesofranges}

Here we look at the properties of landing ranges in more detail.

\subsection{Landing Ranges Lengths for Odd \texorpdfstring{$k$}{k}}

We consider the lengths of landing ranges corresponding to different vertices in the terminal configurations. Given $k^n$ labeled chips at the root and traversal string $t$ of length $n$ and landing range $[a(t, 1), b(t, 1)]$, we define the length of the landing range to be $b(t, 1)-a(t, 1) + 1$, i.e., the number of possible chips that can end up in vertex $v_t$ with traversal string $t$.

One might notice that in Table~\ref{tab:ranges3-3} all landing ranges' lengths are odd. This is not a coincidence, as the following proposition shows.

\begin{proposition}
    For odd $k$, the lengths of the landing ranges are always odd, or in other words, $b(t,1) - a(t,1) + 1$ is odd. 
\end{proposition}

\begin{proof}
    For any $t_j$, when $k$ is odd, the numbers $t_j$ and $k+1-t_j$ have the same parity. Thus, $a(t,1) = \prod_{j=1}^n t_j$ has the same parity as $\prod_{j=1}^n(k+1-t_j)$, which is the same as $b(t,1) = k^n+1 - \prod_{j=1}^n(k+1-t_j)$. It follows that the range length $b(t,1) - a(t,1) + 1$ is odd.
\end{proof}

\subsection{The Lengths of Landing Ranges}

Now, we are back to any $k \geq 2$. We find that the smallest landing range length is 1, corresponding to trivial vertices. They correspond to trivial traversal strings $111\dots 1$ and $kkk\dots k$, see Lemma~\ref{lem:TrivialChips}. We call the landing ranges corresponding to trivial vertices \textit{trivial landing ranges}.

Example~\ref{ex:endsIn112} gave us a landing range $[1,k^{n-1}]$ for the traversal string $111\dots 1112$. We will later show that this is the shortest nontrivial landing range. We continue with another example.

\begin{example}\label{ex:2and5}
Consider a directed binary tree starting with $2^5$ labeled chips. Table~\ref{tab:ranges2-5} shows start, end, and the length of landing ranges for different traversal strings up to a digit permutation.

\begin{table}[h!]\caption{Values for $a(t, 1)$ and $b(t, 1)$ for $k=2, n=5$ \label{tab:SecondThing}}
\centering
\begin{tabular}{|l|l|l|l|l|l|l|l|l|}
\hline
$t$       & 11111 & 11112 & 11122 & 11222 & 12222 & 22222 \\ \hline
$a(t,1)$ & 1 & 2 & 4 & 8 & 16 & 32 \\ \hline
$b(t,1)$ & 1 & 17 & 25& 29& 31& 32\\ \hline
Length & 1 & 16 & 22& 22& 16 & 1 \\ \hline
\end{tabular}
\label{tab:ranges2-5}
\end{table}
\end{example}

The following statement is true due to reflection symmetry.
\begin{proposition}\label{prop:Reflection}
    The lengths of the landing ranges for strings $t$ and $R(t)$ are the same.
\end{proposition}

We noticed that the lengths of landing ranges are relatively large. Here we calculate the smallest and largest possible lengths of an interval $[a(t, 1),b(t, 1)]$ for a nontrivial traversal string $t$ of length $n$.

\begin{theorem}
    The shortest length of a nontrivial landing range is 
    \[k^{n-1},\]
    and it can be achieved for $t$ that consists of all 1's and one 2, or if $t$ consists of all $k$'s and one $k-1$. 
    The longest possible landing range is
    \[k^n+2-k^{\lfloor n/2 \rfloor} -k^{\lceil n/2 \rceil},\]
    and it can be achieved if $t$ consists of $\lfloor n/2 \rfloor$ ones and $\lceil n/2 \rceil$ occurrences of $k$, or vice versa.
\end{theorem}

\begin{proof}
    The landing range for a vertex $v_t$ is $b(t, 1)-a(t, 1) + 1 = k^n+ 2 - \prod_j (k+1-t_j) - \prod_j t_j$. We consider the value
    \[f(t) = \prod_j (k+1-t_j) + \prod_j t_j.\]
    Denote $\frac{\prod_j (k+1-t_j)}{k+1-t_i}$ as $x_i$ and $\frac{\prod_j t_j}{t_i}$ as $y_i$. Then
    \[f(t) = x_i(k+1-t_i) +y_it_i.\]
    We evaluate how $f(t)$ changes if we change $t_i$ by 1. Suppose $t' = t_1t_2\dots t_{i-1}(t_i+1)t_{i+1}\dots t_n$. Then
    \[f(t') - f(t) = y_i - x_i.\]
    If $y_i - x_i > 0$, then the function $f$ increases/decreases when $t_i$ increases/decreases, respectively. If $y_i-x_i < 0$, the trend reverses. If $y_i - x_i = 0$, the function $f$ does not change. It follows that $f(t)$ reaches extreme values when string $t$ consists of 1's and $k$'s. Suppose $t$ has $j$ occurrences of $k$ and $n-j$ occurrences of ones. Then
    \[f(t) = k^{n-j} + k^j.\]
    
    The minimum is reached when $j = \lfloor n/2 \rfloor$.

    The maximum is reached when all the digits in $t$ are either $k$ or 1. However, in this case, we get the trivial landing range. To get the smallest nontrivial landing range, we should be one step away from all $k$'s or all ones. That means the smallest nontrivial landing range is achieved for $t$ that consists of all 1s and one 2. Alternatively, it can be achieved if $t$ consists of all $k$'s and one $k-1$. 
\end{proof}

\begin{example}
    In Example~\ref{ex:2and5}, we see that for $k=2$ and $n=5$, the shortest nontrivial landing range is 16, which is $k^{n-1} = 2^4$, and the longest landing range is 22, which is $k^n+2-k^{\lfloor n/2 \rfloor} -k^{\lceil n/2 \rceil} = 2^5 + 2 - 2^2 - 2^3$.
\end{example}

\section{Where do Chips Land}\label{sec:WheredoChipsLand}

\subsection{General Observation and Notation}

Given that we know for every vertex on the final layer what chips can go there, we can calculate for each $c$ where it can land. The following is the immediate corollary of Theorem~\ref{thm:ranges}.

\begin{corollary}
    Chips $c$ lands at $v_t$ on layer $n+1$ if and only if \[a(t, 1) \leq c \leq b(t, 1).\]
\end{corollary}
From this follows a known fact, stated in Lemma~\ref{lem:TrivialChips}, about the vertices the trivial chips land at.

\begin{example}
    Consider the second-smallest chip. It can only traverse to vertices $v_t$, such that $t$ has $n-1$ ones and 1 two. When we number the vertices, these vertices have numbers 2, $k+1$, $k^2+1$, $\dots$, $k^{n-1} + 1$. If we consider the third smallest chip and $k = 2$ and $n > 1$, we can see that it lands at the same set of vertices.
\end{example}

We can make two observations from this example. 

First, we see that the vertices form groups with the same landing pattern. We discuss these groupings in Section~\ref{sec:groupings}. 

Second, we see that the vertices that an individual chip can land at do not have to be consecutive. However, we are interested in the leftmost and rightmost vertices a chip can land on and the difference between them. We discuss this in Section~\ref{sec:spreads}.

\subsection{Groupings}
\label{sec:groupings}

When one looks at the data, one sees that chips are in consecutive groups that correspond to the same set of final positions. 

Suppose $t$ is the traversal string leading to the terminal layer $n+1$. The set of smallest (or largest) numbers that are possible at such vertices is $a(t,1)$ (or $b(t,1)$). Consider the union $U$ of sets $a(t,1)$ and $b(t,1)$. These numbers divide the chips into groups.

\begin{proposition}
    Consider two consecutive numbers $u_1$ and $u_2$ in the union $U$. Each chip $c$ such that $u_1 \leq c < u_2$ lands on the same subset of vertices $v_t$, for which
    \[a(t, 1) \leq u_1 \quad \textrm{ and } \quad b(t, 1) \geq u_2.\]
\end{proposition}

\begin{proof}
Consider chip $c \in [u_1, u_2)$. Suppose it can land at a vertex $v_t$ in the stable configuration. It follows by Theorem~\ref{thm:ranges} that $c \in [a(t,1), b(t, 1)]$. We also have $a(t, 1) \leq u_1 \leq c$ and $b(t, 1) \geq u_2 > c$. Consider another chip $c' \in [u_1, u_2)$. It follows that
\[a(t, 1) \leq u_1 \leq c \quad \textrm{ and } \quad b(t, 1) \geq u_2 > c,\]
implying that $c'$ can land on the same vertex $v_t$. The statement follows.
\end{proof}

\begin{example}
    For $k=3$ and $n=3$, the numbers of the form $2^i3^j$, where $i+j \leq 3$ are:
\[1,\ 2,\ 3,\ 4,\ 6,\ 8,\ 9,\ 12,\ 18,\ 27.\]
They are the values of $a(t,1)$ for all the vertices on the last layer.

We also subtract these numbers from 28 to obtain the values of $b(t,1)$ for all vertices on the last layer.
\[1,\ 10,\ 16,\ 19,\ 20,\ 22,\ 24,\ 25,\ 26,\ 27.\]
The union of $a(t,1)$ and $b(t,1)$ is
\[\ 1,\ 2,\ 3,\ 4,\ 6,\ 8, \ 9, \ 10,\ 12,\ 16,\ 18,\ 19,\ 20,\ 22,\ 24,\ 25,\ 26, \ 27.\]
The actual data shows that this is exactly where the groupings start.
\end{example}

For $k=2$, we can describe the groups explicitly. 

\begin{proposition}
    Suppose $k=2$. Consider two chips, $c_1$ and $c_2$. Then, the set of vertices chip $c_1$ can land in is the same as that of chip $c_2$ if and only if
    \begin{itemize}
        \item The binary representations of $c_1$ and $c_2$ have the same number of bits that is less than $n$ or
        \item The binary representations of $2^n +1 - c_1$ and $2^n +1- c_2$ have the same number of bits that is less than $n$, or
        \item Chips $c_1$ and $c_2$ are $2^{n-1}$ and $2^{n-1}+1$ in some order.
        \end{itemize}
\end{proposition}

\begin{proof}
If the number of bits in the binary representation of $c$ or $2^n +1 - c$ is 1, then the chip is trivial, and by Lemma~\ref{lem:TrivialChips}, the proposition is satisfied.

Now, assume that neither $c_1$ nor $ c_2$ is a trivial chip. We know that the groupings start at the chips in the union of $a(t,1) = 2^m$, where $m$ is the number of 2s in the traversal string $t$ and $b(t,1) = 2^n - 2^{n-m} +1$, where $m < n$. We know that
    \[a(t,1) \leq 2^{n-1} < b(t',1)\]
    for nontrivial traversal strings $t$ and $t'$.

    Consider chip $c$ and a positive integer $m$ such that
    \[2^m \leq c < 2^{m+1},\]
    where we assume that $m < n-1$. In other words, $m+1 < n$ is the number of digits in the binary representation of $c$. We claim that $c$ lands on vertices with traversal strings that are nontrivial and have $m$ or fewer twos. Suppose $t$ has $m$ or fewer twos, then
    \[a(t,1) \leq 2^m \leq c < 2^{n-1} < b(t, 1).\]
    That means, by Theorem~\ref{thm:ranges}, that $c$ is in the landing range for the vertex $v_t$. If $t$ has more than $m$ twos, then
    \[ c < a(t,1),\]
    which precludes $c$ from landing on $v_t$. It follows that all $c < 2^{n-1}$ with the same number of binary digits less than $n$ land on the same set of vertices. If two chips $c_1, c_2$ less than $2^{n-1}$ have a different number of binary digits, then the set of possible vertices $c_1$ can land on is different from that of $c_2$.

    By symmetry, chips $c > 2^{n-1} + 1$ such that $2^n+1-c$ have the same number of bits less than $n$ in their binary representations have the same landing pattern; and if the number of bits is different, the patterns are different.

    Now, we need to show that two chips $c < 2^{n-1}$ and $c' > 2^{n-1} + 1$ have different landing patterns. Chip $c$ lands on vertices with traversal strings that have $m$ or fewer twos for some $m < n-1$. Chip $c'$ lands on vertices with traversal strings that have $m'$ or fewer ones for some $m' < n-1$. It follows that chip $c$ lands on the second vertex and does not land on the second-to-last vertex. On the other hand, chip $c'$ always lands on the second-to-last vertex and never lands on the second vertex.
    
    The previous discussion did not cover chips $c$ such that both $c$ and $2^n + 1 -c$ have $n$ binary digits. These are chips $2^{n-1}$ and $2^{n-1}+1$, and they can land on and only on nontrivial vertices. Indeed observe that for any nontrivial traversal string $t$, $a(t, 1) \leq 2^{n-1} < b(t, 1)$ and consequently $2^{n-1}, 2^{n-1}+1 \in [a(t, 1), b(t, 1)]$; $2^{n-1}$ and $2^{n-1}+1$ can land on $v_t$ for any nontrivial $t$. Thus, they both land on the same vertices, and this set of vertices is different from the landing sets of other chips.
\end{proof}

\begin{remark}
The number of traversal strings up to permutations is $\binom{n+k-1}{n}$. Thus, the size of the set $U$ is not greater than $2\binom{n+k - 1}{n}$. For large $n$, this number is much smaller than $n^k$. Thus, with a fixed $k$ and increasing $n$, the groups become more noticeable.
\end{remark}

\begin{example}
    For $k=2$, the number of traversal strings up to permutations is $\binom{n+1}{n} = n+1$. We also know that all nontrivial strings $t$, the value $b(t,1)$ are greater than $2^{n-1}$, which is greater than all $a(t,1)$. In addition, for all trivial strings $t$, we have $a(t,1)= b(t,1)$. Thus, the size of the union $U$ is exactly $2n$, which is asymptotically much smaller than $2^n$. Thus, we see $2n-1$ groups.
\end{example}

\subsection{Spreads}
\label{sec:spreads}

Now we look at the leftmost and rightmost vertices that a given chip can land at. We start with some definitions. 

Given a nontrivial chip $c$, let us define $m$ as $\floor{\log_k c}$ and $y$ as $\floor{\frac{c}{k^m}}$. It follows that $yk^m \leq c < (y+1)k^m$. In other words, chip $c$ in $k$-ary has $m+1$ digits and the first digit is $y$.
 
Similarly, given chip $c$, we define $j$ as $\lfloor \log_k(k^n+1-c)\rfloor$ and $x$ as $k+1 - \floor{\frac{k^n+1-c}{k^j}}$. From the definition of $j$, the value $ k^n+1-c$ in $k$-ary has $j+1$ digits with $x$ as the most significant digit. From the definition of $j$ and $x$, it follows that $(k +1 - x) k^j\leq k^n+1-c < (k+2 - x)k^j$.

Note that by definition, we have $0 \leq m, j \leq n-1$; also $x \in \{2, 3, \dots, k\}$ and $y \in \{1, 2, \dots, k-1\}$. In addition, as chip $c$ is nontrivial, if $m=0$, then $y>1$, and if $j=0$, then $x < k$.

\begin{example}\label{ex:mjxy}
    Suppose $k=n=3$ and $c = 11$. Then $m = \floor{\log_k(c)} = \floor{\log_3(11)} = 2$ and $j = \floor{\log_k(k^n+1-c)} = \floor{\log_3(3^3+1-11)} = 2$. This implies that $y = \floor{\frac{c}{k^m}} = \floor{\frac{11}{3^2}} = 1$ and $x = k+1 - \floor{\frac{k^n+1-c}{k^j}} = 4- \floor{\frac{28 - 11}{3^2}} = 3$.
\end{example}

The values, $m$, $j$, $x$, and $y$ play a role throughout this section. We start with a lemma describing how they interact.

\begin{lemma}
\label{lem:nontrivialchip}
    Given a nontrivial chip $c$, either $j=n-1$ or $m=n-1$.
\end{lemma}

\begin{proof}    
    Suppose $\frac{k^n}{2} \leq c \leq k^n-1$. Then $m = \floor{\log_k c} = n-1$. By symmetry, when $2 \leq c \leq \frac{k^n}{2}$, we have $j = n-1$.
\end{proof}

\begin{remark}
The lemma implies that for a nontrivial chip $c$, either the leftmost landing vertex is very close to 1, or the rightmost landing vertex is very close to $k^n$, or both.
\end{remark}

\begin{proposition}\label{prop:LeftRightmostVertex}
Consider $k^n$ chips at the root. Given a nontrivial chip $c \in \{2, 3, \dots, k^n-1\}$, the rightmost vertex it can go to corresponds to the traversal string $t = kk\dots kky11\dots 11$, where the number of $k$'s is $m$. Similarly, the leftmost vertex chip $c$ can go to corresponds to the traversal string $t = 11\dots 11xkk\dots kk$, where the number of $1$s is $j$.
\end{proposition}

\begin{proof}
Suppose $v_t = kk\dots kky11\dots 11$, where the number of $k$'s is $m = \floor{\log_k c}$ and $y = \floor{\frac{c}{k^m}}$. We demonstrate that for any $v_{t'}$ to the right of $v_t$, chip $c$ cannot land at $v_{t'}$. We know that if $v_{t'}$ is to the right of $v_t$, then the first $m$ terms of $t'$ are $k$ and one of the following is true:
\begin{enumerate}
    \item $t'_{m+1} \geq (y+1) = \floor{\frac{c}{y^m}}+1$;
    \item $t'_{m+1}=y= \floor{\frac{c}{y^m}}$ and there exists a $j > m+1$ so that $t'_{j} \geq 2$.
\end{enumerate}
In both cases, we find that $\prod_{j=m+1}^nt'_j \geq y+1$ and $a(t', 1)  = \prod_{j=1}^nt'_j = k^m \prod_{j=m+1}^nt_j' \geq k^m(y+1) > c$. Thus, $c$ cannot land on $v_{t'}$. The argument for the leftmost vertex follows due to the reflection argument.

Now we show that $c$ can land on $v_t$. Observe that $a(t, 1)  = k^m \floor{\frac{c}{k^m}} \leq c$. By symmetry argument $b(t, 1) \geq c$. Thus, by Theorem~\ref{thm:ranges}, chip $c$ can land on $v_t$.
\end{proof}

\begin{corollary}\label{cor:spreadsize}
    Given a nontrivial chip $c \in \{2, 3, \dots, k^n-1\}$, the leftmost vertex it can land at has index $xk^{n-1-j}$; the rightmost vertex has index $1 + (y-1)k^{n-m- 1} +(k-1)\sum_{i=n-m}^{n-1}k^i = 1+(y-1)k^{n-m-1}+k^n - k^{n-m}$. Thus, the spread is equal $2+(y-1)k^{n-m-1}+k^n - k^{n-m} - xk^{n-1-j}$.
\end{corollary}

\begin{example}
    If $k=n=3$ and $c = 11$ as in Example~\ref{ex:mjxy}, where we calculated that $m=j=2$, $y =1$, and $x=3$. It follows that the leftmost vertex corresponds to the traversal string 113, which is vertex number 3; the rightmost vertex corresponds to the traversal string 331, which is vertex number 25. The spread is 23.    
\end{example}

Now we look at the smallest and the second smallest spreads. If $n=1$, then all the spreads are of size 1. Thus, we need only to study the case where $n > 1$. The answer differs for $k=2$ and $k > 2$, so we examine them separately. Moreover, when $k=2$, we always have that $y=1$ and $x=2$.

\begin{theorem}
    Consider $2^n$ labeled chips at the root of a binary (2-ary) tree where $n > 1$. The smallest spread is 1, which occurs on chips 1 and $2^n$. The second smallest spread is $2^{n-1}$, which occurs on chips $2$, $3$, $2^n -2$, and $2^n -1$. The largest spread of vertices that a chip $c$ can land at is $2^n - 2$. This spread occurs on chips $2^{n-1}$ and $2^{n-1}+1$.
\end{theorem}

\begin{proof}
    The chip 1 always lands at the leftmost vertex. Thus, it has a spread of 1. A similar argument works for chip $2^n$.

For the rest of the proof, we assume $c \in \{2, 3, \dots, 2^n-1\}$. We know from Lemma~\ref{lem:nontrivialchip} that either $m = n-1$, or $j = n-1$, or both.

Suppose $m = n-1$, then, from Corollary~\ref{cor:spreadsize}, the spread is
\[2+(y-1)2^{0}+2^n - 2^{1} - x2^{n-1-j} = 2^n - 2k^{n-1-j}.\]

We observe that the second smallest spread equals $2^{n-1}$ and is achieved if $j = 1$. By definition of $j$, this means that
\[1 = \left \lfloor \log_2{(2^n+1-c)} \right \rfloor,\]
implying that $c = 2^n -1$ or $c=2^n-2$. By symmetry, it is also achieved for $c = 2$ or $c=3$.

The largest spread is $2^n-2$ and is achieved for $j=n-1$. By definition of $j$, this means that
\[n-1 = \left \lfloor \log_2{(2^n+1-c)} \right \rfloor,\]
implying that $2^n+1-c \geq 2^{n-1}$. Adding the fact that $m=n-1$, we get that $c \geq 2^{n-1}$. Thus, the largest spread is achieved when $c = 2^{n-1}$ or $c = 2^{n-1} + 1$.
\end{proof}

\begin{example}
    Consider a directed binary tree with $2^4=16$ labeled chips at the root. Here, the second smallest spread a chip can have is $2^{4-1}=8$, which is the spread of chips $2, 3$, $14$, and $15$. The largest spread a chip can have is $2^4-2 = 14$; this is the spread of chips $8$ and $9$.
\end{example}

\begin{theorem}\label{thm:SmallestBiggestSpread}
    Consider $k^n$ labeled chips at the root of a $k$-ary directed tree where $n > 1$ and $k > 2$. The smallest spread is 1, which occurs on chips 1 and $k^n$. The second smallest spread is $k^{n-1}$, which occurs on chips $2$ and $k^n -1$. The largest spread of vertices that a chip $c$ can land at is $k^n-k$. This spread occurs on chips $k^{n-1}$, $k^{n-1}+1$, $k^n - k^{n-1}$, and $k^{n}-k^{n-1}+1$.
\end{theorem}

\begin{proof}
    The chip 1 always lands at the leftmost vertex. Thus, it has a spread of 1. A similar argument works for chip $k^n$. 

For the rest of the proof, we assume $c \in \{2, 3, \dots, k^n-1\}$. We know from Lemma~\ref{lem:nontrivialchip} that either $m = n-1$, or $j = n-1$, or both.

If $m = n-1$, then, from Corollary~\ref{cor:spreadsize}, the spread is
\[2+(y-1)k^{0}+k^n - k^{1} - xk^{n-1-j} = k^n - xk^{n-1-j} -k + y + 1.\]

We observe that the second smallest spread is achieved if $j = 0$ and $x = k-1$. By definition of $j$ and $x$, this means that
\[2 = (k+1 - (k-1)) k^0 \leq k^n+1-c < (k+2 - (k-1))k^0 = 3,\]
implying that $c = k^n -1$. Then $y = \left \lfloor \frac{c}{k^m} \right \rfloor = \left \lfloor \frac{c}{k^{n-1}} \right \rfloor  = 1$. Thus, the smallest spread is
\[k^n - k^{n-1} -k + k-1 + 1 = k^{n-1}.\]
By symmetry, the same spread is also achieved for $c = 3$.

The largest spread is achieved when $j=n-1$. As we mentioned in the definition of $x$, in the case of $j=n-1$, the smallest possible $x$ is 2. This means that
\[k^n - k^{n-1} = (k+1 - 2) k^{n-1} \leq k^n+1-c < (k+2 - 2)k^{n-1} = k^n.\]
Hence $1 < c \leq k^{n-1}+1$. As we assumed that $\log_k c = m= n-1$, we get $k^{n-1} \leq c \leq k^{n-1}+1$. It follows that $y = \floor{\frac{c}{k^{n-1}}} =1$, and we find that the largest spread is
\[k^n - xk^{n-1-j} -k + y + 1  = k^n - 2k^{n-1-(n-1)} -k + 1 +1 = k^n - k.\]
By symmetry, it is also achieved for $k^n - k^{n-1} \leq c \leq k^{n}-k^{n-1}+1$.
\end{proof}

\begin{remark}
    Notice that the smallest spread of 1 is achieved for 2 values of $c$. The second smallest spread is achieved for 2 values of $c$ when $k > 2$ and for 4 values of $c$ when $k=2$. The largest smallest spread is achieved for 4 values of $c$ when $k > 2$ and for 2 values of $c$ when $k = 2$.
\end{remark}

\begin{example}
Consider a directed $3$-ary tree starting with $3^4$ labeled chips at the root. The smallest spread is 1, which occurs on chips $1$ and on $81$. Theorem~\ref{thm:SmallestBiggestSpread} tells us that second smallest spread is $3^{4-1}=27$, which occurs on chips $2$ and 80. In addition, the largest spread is $3^{4}-3=78$, which occurs on chips $3^{4-1}=27$, $3^{4-1}+1=28$, $3^4 - 3^{4-1}=54$, and $3^4 - 3^{4-1} + 1 =55$. 
\end{example}

We showed restrictions on how individual chips can land. We conclude with an example where each chip follows the restrictions, but the terminal configuration is not possible.

\begin{example}
    Suppose $k=2$ and $n=3$. Consider the configuration where the chips on layer 4 read from left to right are $15243678$. We observe that for each traversal string $t$, the chip at $v_t$ is within the landing range of $v_t$. However, this configuration cannot be attained from any chip-firing strategy. To see this, we observe that during the firing process, $0, 4$ must have been fired to $v_{11}$, and $1, 3$ must have been fired to $v_{12}$. We note that by applying Proposition~\ref{prop:Layer2} on the subtree rooted at $v_1$, the largest chip on $v_{11}$, which is $4$, must be smaller than that on $v_{12}$, which is $3$. We obtain a contradiction.
\end{example}

\section{Acknowledgments} 

We thank Professor Alexander Postnikov for suggesting the topic of labeled chip-firing on directed trees and helping formulate the proposal of this research problem, and for helpful discussions.

The first and second authors are employed by the MIT Department of Mathematics.

All figures in this paper were generated using TikZ.

\smallskip

\noindent
Ryota Inagaki \\
\textsc{
Department of Mathematics, Massachusetts Institute of Technology\\
77 Massachusetts Avenue, Building 2, Cambridge, Massachusetts, U.S.A. 02139}\\
\textit{E-mail address: }\texttt{inaga270@mit.edu}
\medskip

\noindent
Tanya Khovanova \\
\textsc{
Department of Mathematics, Massachusetts Institute of Technology\\
77 Massachusetts Avenue, Building 2, Cambridge, MA, U.S.A. 02139}\\
\textit{E-mail address: }\texttt{tanyakh@yahoo.com}
\medskip

\noindent
Austin Luo \\
\textsc{
Morgantown High School,\\
109 Wilson Ave, Morgantown, West Virginia, U.S.A. 26501}\\
\textit{E-mail address: }\texttt{austinluo116@gmail.com}
\medskip

\end{document}